\documentclass%[draft]
{amsart}
\usepackage{xcolor}
\usepackage{hyperref}
\newcommand{\R}{\mathbb{R}}

\newcommand{\Z}{\mathbb{Z}}

\newcommand{\sfu}{\dot{\mathsf{V}}}
\newcommand{\sfv}{\mathsf{V}}
\newcommand{\sfk}{\mathsf{K}}
\newcommand{\sn}{\check{\mathbb{R}}}
\newcommand{\rem}{\varrho}

\newtheorem{lemma}{Lemma}

\newtheorem{theorem}{Theorem}

\theoremstyle{definition}
\newtheorem{definition}{Definition}
\newtheorem{example}{Example}
\newtheorem{remark}{Remark}

\begin{document}

\title{Canonical Cauchy sequences for real numbers}
\author{Rinat Kashaev}
\address{Section de math\'ematiques, Universit\'e de Gen\`eve,
2-4 rue du Li\`evre, 1211 Gen\`eve 4, Suisse\\}
\email{rinat.kashaev@unige.ch}
\date{October 12, 2020}

\begin{abstract}
 Based on continued fractions with subtractions, we identify the set of real numbers with the set of infinite integer sequences with all terms but the first one greater or equal to two. Each such sequence produces in a canonical way a unique strictly decreasing Cauchy sequence of rationals which converges to the corresponding real number. The correspondence is such that the standard order of real numbers is translated to the lexicographic order of sequences.
\end{abstract}
\maketitle
\section{Introduction}
According to the common definition,  a real number is an equivalence class in the set of Cauchy sequences of rational numbers, see for example~\cite{MR0038404,MR2004919}. Being theoretically proper, this definition is not convenient for practical calculations as each equivalence class is an infinite set. Thus, solving the problem of choosing in some canonical way a unique representative from each equivalence class would not only improve the practical utility of Cauchy sequences, but  also would provide us with a simpler definition of real numbers. 

The existing descriptions of real numbers by infinite decimals or simple continued fractions solve the problem only partially since the representing Cauchy sequences are unique for irrational numbers while there are two representatives for infinitely many rational numbers. A convention on an extra choice between two representatives for those rational numbers is not always satisfactory since it could complicate recursive algebraic calculations that use only finite but arbitrarily large parts of the representing Cauchy sequences. 
In the case of simple continued fractions, one can also remark that the rational numbers are distinguished by finiteness of their continued fractions as opposed to infinite continued fractions of irrational numbers~\cite{MR0146146}.

Motivated by these considerations, in this paper, we suggest a simple modification of simple continued fractions into ``subtraction continued fractions'' by changing all additions to subtractions. This allows us to construct a natural right inverse to the projection map from the set of Cauchy sequences of rational numbers to the set of real numbers. 

\subsection*{Outline} In Section~\ref{sec:scf}, we introduce a set $\sn$ of s-numbers. These are infinite integer sequences with all terms but the first one greater or equal to two. Then,  by using the subtraction continued fractions, we construct a map $ \imath\colon \R\to\sn$.
 In Section~\ref{cauchy-sequences}, to each  s-number, we associate a bounded from below strictly decreasing sequence of rational numbers called right convergents (Theorem~\ref{theorem1}).  The limit of that sequence gives us the inverse map $\imath^{-1}\colon \sn\to\R$ (Theorem~\ref{theorem2}). In Theorem~\ref{theorem3}, we prove that the bijection $ \imath$ relates the standard order of $\R$ to the lexicographic order of sequences. In the final Section~\ref{final-sec3}, we develop a matrix technique for recursive conversion of simple continued fractions to subtraction continued fractions and vice versa.

\section{Subtraction continued fractions}\label{sec:scf}
 \subsection{Integer sequences and s-numbers}
 We denote $\omega:=\Z_{\ge0}$ (the first infinite ordinal).
 % and, by abus of notation, the elements of $\omega$  are also interpreted as finite ordinals:  \begin{equation}0=\{\}=\emptyset,\quad 1=\{0\},\quad \dots,\quad n=\{0,1,\dots,n-1\}\quad \forall n\in\omega. \end{equation}
 For  sets $A$ and $B$, $B^A$ denotes the set of all maps $f\colon A\to B$. In particular, $\Z^\omega$ denotes the set of all integer sequences so that, for any $s\in\Z^\omega$, we write $s_n$ instead of $s(n)$.
\begin{definition}
 An  \emph{\color{blue} s-number} is an element $s\in \Z^\omega$ such that  $s_n\ge2$ if $n>0$. The set of all s-numbers is denoted $\sn$. For any $n\in\omega$, the integer $s_n$ is called $n$-th \emph{\color{blue} (partial) quotient} of $s$.
\end{definition}

For any $n\in\omega$, we associate the $n$-th \emph{\color{blue} remnant}
\begin{equation}
\rem_n\colon \sn\to\sn,\quad \rem_n(s)=\rem_ns,\quad (\rem_ns)_k=s_{n+k}\quad \forall k\in\omega.
\end{equation}
In other words, the $n$-th remnant $\rem_ns$ of $s\in\sn$ is obtained from $s$ by deleting the first $n$ numbers in the sequence $(s_0,s_1,\dots)$.  Thus, for any $n\in\omega$, we can write an equality
\begin{equation}\label{split-s-n}
s=(s_0,s_1,s_2,\dots, s_{n-1},\rem_ns)
 \end{equation}
if $n>0$ and $s=\rem_0s$ otherwise. 
\subsection{Subtraction continued fraction}
We define a map $\imath\colon\R\to\sn$  called \emph{\color{blue} subtraction continued fraction} by associating to a real number $x$ the s-number $\imath x$ constructed by the following recursive procedure.

Define $r_0x:=x$. There exists a unique integer $ x_0$ such that
 \begin{equation}
  1\ge  x_0-r_0x>0\ 
\Leftrightarrow \  1\le \frac1{ x_0-r_0x}<\infty.
\end{equation}
Assume that, for an integer $k\in\omega$, we have constructed tuples 
\begin{equation}
( x_0, x_1,\dots, x_{k})\in \Z^{k+1},\quad
(r_0x,\dots,r_{k}x)\in \R^{k+1}
\end{equation}
such that 
\begin{equation}
1\ge  x_i-r_ix>0\quad \forall i\le k,\quad r_{i+1}x=\frac1{ x_i-r_ix}\quad  \forall i< k
\end{equation}
where there is no the second condition if $k=0$. Then we define
\begin{equation}
r_{k+1}x:=\frac1{ x_{k}-r_{k}x}\ge 1
\end{equation}
and there exists a unique integer $ x_{k+1}$ such that $1\ge x_{k+1}-r_{k+1}x>0$. 
In this way, we obtain
an s-number $ \imath x\in\sn$ defined by
\begin{equation}
 ( \imath x)_n= x_n\quad\forall n\in\omega.
\end{equation}
\subsection{Finite subtraction continued fractions}
For any $n\in \omega$, our recursive procedure for the subtraction continued fraction produces a map
\begin{equation}
r_n\colon\R\to\R,\quad x\mapsto r_nx
\end{equation}
which we also call $n$-th  \emph{\color{blue} remnant}. Notice that the remnants of real numbers are real numbers, while the remnants of s-numbers are s-numbers, and the map $ \imath$ intertwines the two types of remnants:
\begin{equation}
\rem_n\circ \imath = \imath\circ r_n\quad \forall n\in\omega.
\end{equation}

For any $x\in\R$ and $n\in\omega$, we have an equality 
\begin{equation}\label{cuts}
 x=\langle  x_0 , x_1, x_2,\dots, x_{n-1} ,r_nx\rangle:= x_0-\cfrac1{ x_1-\cfrac1{x_2-\cfrac1{\ddots -\frac1{ x_{n-1}-\frac1{r_nx}}}}}
\end{equation}
if $n>0$ and $x=r_0x$ otherwise, and we call it $n$-th \emph{\color{blue} finite subtraction continued fraction} of $x$, see~\cite{MR0146146} for a general definition of continued fractions.

Let us stress here that the subtraction continued fraction of a real number, being an s-number, is always an infinite sequence, while finite sequences are understood either in the sense of equalities~\eqref{cuts},  where the last term is a real number not necessarily an integer, or equalities~\eqref{split-s-n} where the last term is an s-number.

\begin{example} The subtraction continued fraction of an integer $m\in\Z$ is of the form
\begin{equation}
 \imath m=( m+1,2,2,2,\dots)
\end{equation}
that is an s-number such that $(\imath m)_n=m_n=2$ if $n>0$ and $m+1$ otherwise. In particular, $ 1_n=2$ for all $n\in\omega$. We also have the remnants $r_nm=1$ for all $n>0$. Equalities \eqref{split-s-n} for $ \imath m$ take the form
\begin{equation}\label{imath mn}
 \imath m=(m+1,\underbrace{2,\dots,2}_{n-1 \text{ times}}, \imath 1)\quad \forall n>0,
\end{equation}
while equalities~\eqref{cuts} with $x=m$ take the form
\begin{equation}\label{infty-ineqs-m}
m=\langle m+1 ,\underbrace{2,\dots,2}_{n-1 \text{ times}} ,1\rangle \quad \forall n>0.
\end{equation}
  More generally, a real number $x$ is rational if and only if there exists $N\in\omega$ such that $ x_n=2$ for all $n\ge N$ or, equivalently,  $r_nx=1$ for all $n\ge N$.
 \end{example}
\begin{remark}
As the number  $2$ plays a distinguished role for s-numbers, we will abbreviate by $\widehat{k}$ a sequence of $2$'s of length $k\in\omega$. With this notation, equalities in~\eqref{imath mn} and \eqref{infty-ineqs-m} read
  \begin{equation}
 \imath m=(m+1,\widehat{n-1}, \imath1),\quad m=\langle m+1 ,\widehat{n-1} ,1\rangle\quad \forall n>0.
\end{equation}
\end{remark}
\section{Canonical Cauchy sequences}\label{cauchy-sequences}
\subsection{M\"obius transformations}
We will use the standard M\"obius action of the group $ \operatorname{GL}_2\R$ on the extended real line $\widehat\R:=\R\cup\{\infty\}$:
\begin{equation}
g(x)=\frac{ax+b}{cx+d}, \quad g=\begin{pmatrix}
 a&b\\
 c&d
\end{pmatrix}\in \operatorname{GL}_2\R,\quad x\in\widehat\R.
\end{equation}
\begin{lemma}\label{g(t)}
Let a matrix $g=\left(
\begin{smallmatrix}
 a&b\\
 c&d
\end{smallmatrix}\right)\in \operatorname{SL}_2\!\R$ be such that
\begin{equation}\label{abcd-ineqs}
c>0,\quad  c+ d>0.
\end{equation}
Then, for  all $t\ge1$, $g(t)=\frac{at+b}{ct+d}$ is a well defined strictly growing function satisfying the inequalities
\begin{equation}
0<\frac ac-g(t)\le \frac1{(c+d)c},\quad 0\le g(t)-\frac {a+b}{c+d}<\frac1{(c+d)c}.
\end{equation}
\end{lemma}
\begin{proof}
Let $t\ge1$.  The denominator of $g(t)$ is strictly positive:
\begin{equation}
ct+d\ge c+d>0,
\end{equation}
as well as its derivative:
\begin{equation}
\frac{\partial g(t)}{\partial t}=\frac{(ct+d)a-(at+b)c}{(ct+d)^2}=\frac{1}{(ct+d)^2}>0.
\end{equation}
Thus, we obtain inequalities
\begin{equation}
\frac{a+b}{c+d}=g(1)\le g(t)<\lim_{t'\to+\infty}g(t')=\frac{a}{c}
\end{equation}
which imply that
\begin{equation}
0<\frac ac-g(t)\le\frac ac-\frac{a+b}{c+d}=\frac1{(c+d)c}
\end{equation}
and 
\begin{equation}
0\le g(t)-\frac{a+b}{c+d}<\frac ac-\frac{a+b}{c+d}=\frac1{(c+d)c}.
\end{equation}
\end{proof}
\subsection{Left/Right convergents}
For any pair $(n,s)\in\omega\times\sn$, we associate  a matrix $g_{n,s}\in\operatorname{SL}_2\!\Z$ as follows
\begin{equation}
g_{n,s}=
\begin{pmatrix}
 a_{n,s}&b_{n,s}\\
 c_{n,s}&d_{n,s}
\end{pmatrix}=
\sfv(s_0)\sfv(s_1)\cdots\sfv(s_n),\quad \sfv(m):=
\begin{pmatrix}
 m&-1\\
 1&0
\end{pmatrix}.
\end{equation}
The matrix $\sfv(2)$ is distinguished by the fact that it is parabolic so that there is an explicit formula for its $k$-th power
\begin{equation}\label{v2k}
\sfv(2)^k=\begin{pmatrix}
 1+k&-k\\
 k&1-k
\end{pmatrix}\quad\forall  k\in\Z.
\end{equation}
From the definition of the subtraction continued fraction of a real number $x\in\R$, it follows that
\begin{equation}\label{x=gnx(x)}
x=g_{n, \imath x}(r_{n+1}x)\Leftrightarrow r_{n+1}x=g_{n, \imath x}^{-1}(x)\quad \forall n\in\omega.
\end{equation}
Indeed, we have
\begin{multline}
x=\langle  x_0 ,r_1x\rangle= x_0-\frac1{r_1x}
=\sfv( x_0)(r_1x)\\=\sfv( x_0)(\sfv( x_1)(r_2x))=\cdots=g_{n, \imath x}(r_{n+1}x).
\end{multline}

\begin{definition}
Let $s\in\sn$. For any $n\in\omega$,  the $n$-th \emph{\color{blue} (right) convergent} of $s$ is the rational number $R_n(s):=\frac{a_{n,s}}{c_{n,s}}$, the $n$-th \emph{\color{blue} left convergent} of $s$ is the rational number $L_n(s):=\frac{a_{n,x}+b_{n,s}}{c_{n,x}+d_{n,s}}$ and the  $n$-th \emph{\color{blue} accuracy} of $s$ is the integer $A_n(s):=(c_{n,s}+d_{n,s})c_{n,s}$. The denominators in these definitions are non zero due to Theorem~\ref{theorem1}, see below.
\end{definition}

\begin{remark}\label{rem-rec-rel}
 We have the relations
 \begin{equation}\label{recrelLR}
R_{n}(s)=g_{n-1,s}(s_{n}),\quad L_{n}(s)=g_{n-1,s}(s_{n}-1) \quad \forall n\in\omega
\end{equation}
with the convention $g_{-1,s}=\left(
\begin{smallmatrix}
 1&0\\0&1
\end{smallmatrix}\right)$. Indeed, denoting $g_{n-1,s}=:g=\left(
\begin{smallmatrix}
 a&b\\ c&d
\end{smallmatrix}\right)$, we have
\begin{equation}
g_{n,s}=g\sfv(s_n)=
\begin{pmatrix}
 as_n+b&-a\\ cs_n+d&-c
\end{pmatrix}\Rightarrow R_n(s)=g(s_n),\quad L_n(s)=g(s_n-1).
\end{equation}

As a consequence, an s-number $s$ can be recovered from the sequence of its right or left convergents by simple recursive procedures:
 \begin{equation}
s_{n}=g_{n-1,s}^{-1}(R_{n}(s))=1+g_{n-1,s}^{-1}(L_{n}(s))\quad \forall n\in\omega.
\end{equation}
\end{remark}

\begin{theorem}\label{theorem1}
 Let $s\in\sn$. For any $n\in\omega$, one has the inequalities 
\begin{equation}\label{ineqs}
c_{n+1,s}>c_{n,s}>0,\quad   c_{n+1,s}+ d_{n+1,s}\ge c_{n,s}+ d_{n,s}>0,\quad  d_{n,s}\le0
\end{equation}
and the equality
\begin{equation}\label{Rn-Ln=1/A}
R_n(s)-L_n(s)=\frac1{A_n(s)}.
\end{equation}
 The sequence of accuracies  $(A_n(s))_{n\ge0}$ is a strictly increasing sequence of positive integers. The  sequence of right convergents $(R_n(s))_{n\ge0}$ is a strictly decreasing sequence of rationals bounded from below.  
The sequence of left convergents $(L_n(s))_{n\ge0}$ is a non-strictly increasing sequence of rationals bounded from above. 
\end{theorem}
\begin{proof}
 Inequalities~\eqref{ineqs} are verified by induction on $n$. 
 
 For $n=0$, we have the inequalities
 \begin{equation}
c_{0,s}=1>0,\quad c_{0,s}+d_{0,s}=1>0,\quad  d_{0,s}=0\le0.
\end{equation}
Let $n\in\omega$ and suppose that the matrix $g_{n,s}=:g=\left(
\begin{smallmatrix}
 a&b\\
 c&d
\end{smallmatrix}\right)\in \operatorname{SL}_2\!\Z$ satisfies inequalities~\eqref{abcd-ineqs} and $d\le0$.
Then, with $m:=s_{n+1}\ge 2$, we have
\begin{equation}\label{gVm}
g_{n+1,s}=g \sfv(m)=\begin{pmatrix}
 a& b\\
 c& d
\end{pmatrix}\begin{pmatrix}
m& -1\\
1& 0
\end{pmatrix}=
\begin{pmatrix}
 ma+b& -a\\
 mc+d& -c
\end{pmatrix}=:\begin{pmatrix}
 a'& b'\\
 c'& d'
\end{pmatrix}
\end{equation}
so that
\begin{multline}
c'=mc+d=c+\underbrace{(m-2)c}_{\ge0}+\underbrace{c+d}_{>0}>c>0,\\
c'+d'=(m-1)c+d=c+d+\underbrace{(m-2)c}_{\ge0}\ge c+d>0,\quad
d'=-c< 0\le0,\quad 
\end{multline}
thus completing the induction on $n$.

Equality~\eqref{Rn-Ln=1/A} is the identity 
\begin{equation}
\frac ac-\frac{a+b}{c+d}=\frac1{(c+d)c}
\end{equation}
valid for any matrix $\left(
\begin{smallmatrix}
 a&b\\
 c&d
\end{smallmatrix}\right)\in \operatorname{SL}_2\!\R$ if the denominators are non zero. 

The strict increasing property of the sequence $(A_n(s))_{n\ge0}$ directly follows from inequalities~\eqref{ineqs}.
 
 The strict decreasing property of the sequence $(R_n(s))_{n\ge0}$ is a consequence of Lemma~\ref{g(t)} and the definitions of $a'$ and $c'$ in~\eqref{gVm}:
\begin{equation}
R_{n+1}(s)=\frac{a'}{c'}=g(m)<\lim_{t\to+\infty}g(t)=\frac{a}{c}=R_{n}(s).
\end{equation}
Likewise, the non-strict increasing property of the sequence  $(L_n(s))_{n\ge0}$ is a consequence of Lemma~\ref{g(t)}:
\begin{equation}
L_{n+1}(s)=\frac{a'+b'}{c'+d'}=g(m-1)\ge g(1)=\frac{a+b}{c+d}=L_{n}(s)
\end{equation}
with the notation in~\eqref{gVm}.

Together with the increasing property of $(L_n(s))_{n\ge0}$, the decreasing property of  $(R_n(s))_{n\ge0}$ and the strict positivity of the accuracies, equalities~\eqref{Rn-Ln=1/A} imply that
\begin{equation}
R_n(s)=L_n(s)+\frac1{A_n(s)}> L_0(s),\quad L_n(s)=R_n(s)-\frac1{A_n(s)}<R_0(s),
\end{equation}
that is $(R_n(s))_{n\ge0}$ is bounded from below, and $(L_n(s))_{n\ge0}$ is bounded from above.

\end{proof}
\begin{remark}\label{rem-Rns-full-info}
 For any s-number $s$ and any index $n\in\omega$, the $n$-th right convergent $R_n(s)$ contains the full information about the matrix $g_{n,s}$ and, in particular, the $n$-th left convergent and the $n$-th accuracy. Indeed, the integers $a_{n,s}$ and $c_{n,s}$ are mutually prime with $c_{n,s}>0$, so that they can be extracted from the rational number $R_n(s)$  as  the signed numerator and the positive denominator respectively. Then, $b_{n,s}$ and $d_{n,s}$ are uniquely determined from the determinant condition
 \begin{equation}
a_{n,s}d_{n,s}-b_{n,s}c_{n,s}=1
\end{equation}
and the inequalities $-c_{n,s}<d_{n,s}\le 0$. 
\end{remark}
\begin{example}[The right convergents of an integer]
 Let us calculate the right convergents for an integer $m\in\Z$ given by the sequence $\langle m+1 ,1\rangle$.
 By using the explicit formula~\eqref{v2k}, the associated matrices read
 \begin{multline}
g_{n,m}=\sfv(m+1)\sfv(2)^n=\begin{pmatrix}
 m+1&-1\\
 1&0
\end{pmatrix}\begin{pmatrix}
 1+n&-n\\
 n&1-n
\end{pmatrix}\\
=\begin{pmatrix}
(n+1) m+1&-mn-1\\
 n+1&-n
\end{pmatrix}
\end{multline}
so that
\begin{equation}
R_n(m)=m+\frac1{n+1} \quad \forall n\in\omega.
\end{equation}
\end{example}
\begin{example}[The right convergents of $\frac12$] We have the sequence $\frac12=\langle 1 ,3 ,1\rangle$ so that
\begin{equation}
g_{0,\frac12}=\sfv(1)=
\begin{pmatrix}
 1&-1\\
 1&0
\end{pmatrix}\Rightarrow R_0(1/2)=1,
\end{equation}
\begin{equation}
g_{1,\frac12}=\sfv(1)\sfv(3)=\begin{pmatrix}
 2&-1\\
 3&1
\end{pmatrix}\Rightarrow R_1(1/2)=\frac 23,
\end{equation}
and, for any $n>0$,
\begin{equation}
g_{n,\frac12}=g_{1,\frac12}\sfv(2)^{n-1}=
\begin{pmatrix}
 1+n&-n\\
 1+2n&1-2n
\end{pmatrix}\Rightarrow  R_{n}(1/2)=\frac{1+n}{1+2n}
\end{equation}
which also reproduces the correct answer for $n=0$.
 
\end{example}
\subsection{Invertibility of the subtraction continued fraction} Given an s-number $s\in\sn$,
Theorem~\ref{theorem1} implies that the sequences of left and right convergents of $s$ converge to one and the same real number denoted $\jmath s$.

\begin{theorem}\label{theorem2}
The map $\jmath\colon\sn\to\R$ defined by
\begin{equation}
\jmath(s)=\jmath s:=\lim_{n\to\infty}R_n(s)
\end{equation}
 is the inverse of $ \imath\colon\R\to\sn$.
\end{theorem}
\begin{proof}
 {\color{purple}($\jmath\circ \imath=\operatorname{id}_{\R}$)} 
 For any real number $x\in\R$,  the sequences of left and right convergents of $\imath x$ satisfy the inequalities
 \begin{equation}\label{C_n(x)-xineqs}
0<R_n( \imath x)-x\le \frac1{A_{n}( \imath x)}\quad\forall n\in\omega
\end{equation}
and
\begin{equation}\label{L_n(x)-xineqs}
0\le x-L_n( \imath x)<\frac1{A_{n}( \imath x)}\quad\forall n\in\omega.
\end{equation}
which follow from Lemma~\ref{g(t)} with $g=g_{n, \imath x}$ and $t=r_{n+1}x$ and by taking into account equality~\eqref{x=gnx(x)}. 
The fact that both sequences converge to $x$ follows from these inequalities and Theorem~\ref{theorem1}: the sequence of accuracies $(A_n( \imath x))_{n\ge0}$ is a strictly increasing sequence of positive integers so that the sequence of their reciprocals converges to zero.

 {\color{purple}($\imath\circ\jmath=\operatorname{id}_{\sn}$)}  
Given $s\in\sn$, denote $x:=\jmath s$ and $s':= \imath x$. By Theorem~\ref{theorem1}, we have two families of inequalities
\begin{equation}\label{LxRp}
L_n(s)\le x<R_n(s),\quad L_n(s')\le x<R_n(s')\quad \forall n\in\omega
\end{equation}
which for $n=0$ reduce to
 \begin{equation}
s_0-1\le x<s_0,\quad s'_0-1\le x<s'_0\Leftrightarrow s_0=s'_0=x_0.
\end{equation}
Proceeding by induction, assume that for $k\in\omega$, we have 
\begin{equation}
s_i=s'_i\quad \forall i\le k.
\end{equation}
Then, we have $g_{k,s}=g_{k,s'}=:g$. Denoting $m:=s_{k+1}$ and  $m':=s'_{k+1}$ and taking into account Lemma~\ref{g(t)} and the relations for the convergents in~\eqref{recrelLR}, inequalities~\eqref{LxRp} for $n=k+1$ take the form
\begin{multline}
g(m-1)\le x<g(m),\quad g(m'-1)\le x<g(m')\\
\Leftrightarrow m-1\le g^{-1}(x)<m,\quad m'-1\le g^{-1}(x)<m'
\Leftrightarrow m=m'=(g^{-1}(x))_0.
\end{multline}
We conclude that $s=s'$.
\end{proof}
\subsection{Orders in $\R$ and $\sn$}
Let us identify $\R$ and $\sn$ through the subtraction continued fraction.
\begin{theorem}\label{theorem3}
The order of $\R$  corresponds to the lexicographic order of $\sn$. This means that $x<y$ if and only if there exists $N\in\omega$ such that $x_N<y_N$ and $x_n=y_n$ for all $n<N$.
 \end{theorem}
\begin{proof} Assuming $x\ne y$, let $N\in\omega$ be such that $x_N\ne y_N$ and $x_n=y_n$ for all $n<N$.
 We have
\begin{equation}\label{x=gVy=gV}
x=g\sfv(x_N)(r_{N+1}x),\quad y=g\sfv(y_N)(r_{N+1}y),\quad g:=g_{N-1,x}=g_{N-1,y}
\end{equation}
 with the convention $g_{-1,z}=\left(
\begin{smallmatrix}
 1&0\\
 0&1
\end{smallmatrix}\right)
$ for any $z\in\R$.  By using the action 
\begin{equation}
\sfv(s)(t)=\frac{s t-1}{t}=s-\frac1t, 
\end{equation}
we rewrite~\eqref{x=gVy=gV}
\begin{equation}
x=g\Big(x_N-\frac1{r_{N+1}x}\Big),\quad  y=g\Big(y_N-\frac1{r_{N+1}y}\Big).
\end{equation}
The fact that, due to Lemma~\ref{g(t)}, the function $g(t)$ is strictly growing for $t\ge1$, we obtain the equivalences
\begin{equation}
x<y\Leftrightarrow x_N-\frac1{r_{N+1}x}<y_N-\frac1{r_{N+1}y}\Leftrightarrow \frac1{r_{N+1}y}-\frac1{r_{N+1}x}<m
\end{equation}
where $m:=y_N-x_N\in\Z_{\ne0}$.
Complemented with the implication
\begin{equation}
 0<\frac1{r_{N+1}y}\le 1,\quad -1\le -\frac1{r_{N+1}x}<0\Rightarrow -1<\frac1{r_{N+1}y}-\frac1{r_{N+1}x}<1,
\end{equation}
we obtain the equivalence
\begin{equation}
x<y\Leftrightarrow -1<m.
\end{equation}
Now, a non zero integer $m$ which is strictly greater than $-1$ is a strictly positive integer. Thus,
\begin{equation}
x<y\Leftrightarrow 0<m=y_N-x_N\Leftrightarrow x_N< y_N.
\end{equation}
\end{proof}

\section{Relating simple and subtraction continued fractions}\label{final-sec3}
\subsection{Simple continued fractions}
Simple continued fractions are based on the matrix valued function
\begin{equation}
\sfu\colon \Z\to \operatorname{GL}_2\!\Z,\quad \sfu(m)=\begin{pmatrix}
 m&1\\
 1&0
\end{pmatrix}
\end{equation}
in the same way as subtraction continued fractions are based on the matrix valued function
\begin{equation}
\sfv\colon \Z\to \operatorname{SL}_2\!\Z,\quad\sfv(m)=\begin{pmatrix}
 m&-1\\
 1&0
\end{pmatrix}
\end{equation}
as is described in Section~\ref{sec:scf}.
Indeed, the  simple continued fraction of an irrational number $x\in\R$ is the expression~\cite{MR0146146}
\begin{equation}
x=[\dot x_0 ,\dot x_1,\dot x_2,\dots]:=
\dot x_0+\frac1{\dot x_1+\frac1{\dot x_2+\frac1{\cdots}}},
\quad \dot x_i\in\Z\quad\text{and}\quad 
   \dot x_i\ge1\quad \text{if}\quad i>0,
\end{equation}
where we put a dot over $x_i$ in order to distinguish it from the $x_i$ in the subtraction continued fraction.
Similarly to the equalities~\eqref{cuts} and \eqref{x=gnx(x)}, we have
\begin{equation}\label{x=dotgnx(x)}
x=[\dot x_0 ,\dot x_1,\dots, \dot x_n ,\dot r_{n+1}x]=\dot g_{n,x}(\dot r_{n+1}x)\quad \forall n\in\omega
\end{equation}
where $\dot r_kx\in\R_{>1}$ for all $k>0$ and
 \begin{equation}
\dot g_{n,x}:=\sfu(\dot x_0)\sfu(\dot x_1)\cdots \sfu(\dot x_n).
\end{equation}
\subsection {Some identities in $\operatorname{GL}_2\Z$}
 A number of algebraic properties of the maps $\sfu$ and $\sfv$ allow us to relate the two types of continued fractions.
\begin{lemma}
One has the following identities
\begin{equation}\label{uuu=u}
\sfu(x)\sfu(0)\sfu(y)
=\sfu(x+y)
\end{equation}
\begin{equation}\label{uu=vuu}
\sfu(x)\sfu(y)
=\sfv(x+1)
\sfu(1)\sfu(y-1),
\end{equation}
\begin{equation}\label{vv=uuv}
\sfv(x)\sfv(y)
=\sfu(x-1)
\sfu(1)\sfv(y-1),
\end{equation}

\end{lemma}
\begin{proof}
Relation~\eqref{uuu=u} is verified by a direct calculation. 
Although the two other relations can also be verified by direct calculations, it is instructive 
to introduce an auxiliary matrix of order $2$
\begin{equation}
\sfk:=\begin{pmatrix}
 1&0\\
 1&-1
\end{pmatrix}, \quad \sfk^2=\begin{pmatrix}
 1&0\\
 0&1
\end{pmatrix},
\end{equation}
which satisfies the easily verifiable relations
 \begin{equation}\label{uk=v}
\sfu(x)\sfk=\sfv(x+1),
\end{equation}
\begin{equation}\label{ku=uu}
\sfk\sfu(x)=\sfu(1)\sfu(x-1),
\end{equation}
\begin{equation}\label{kv=uv}
\sfk\sfv(x)=\sfu(1)\sfv(x-1).
\end{equation}
Now, relation~\eqref{uu=vuu} is proven in two steps:
\begin{equation}
\sfu(x)\sfu(y)
=\sfv(x+1)\sfk\sfu(y)=\sfv(x+1)
\sfu(1)\sfu(y-1)
\end{equation}
where, in the first equality, we use \eqref{uk=v} and, in the second equality, \eqref{ku=uu}.
Likewise for relation~\eqref{vv=uuv}:
\begin{equation}
\sfv(x)\sfv(y)
=\sfu(x-1)\sfk\sfv(y)=\sfu(x-1)\sfu(1)\sfv(y-1)
\end{equation}
where, in the first equality, we use \eqref{uk=v} and, in the second equality, \eqref{kv=uv}.
\end{proof}
\subsection{Relating the two types of continued fractions}
Formulae~\eqref{uuu=u}--\eqref{vv=uuv} allow us to recursively transform  simple and subtraction continuous fractions to each other. For example, starting from a  simple continued fraction, we calculate
\begin{multline}
x=[\dot x_0 ,\dot x_1 ,\dot r_2x]=\sfu(\dot x_0)\sfu(\dot x_1)(\dot r_2x)=\sfv(\dot x_0+1)\sfu(1)\sfu(\dot x_1-1)(\dot r_2x)\\
=\langle \dot x_0+1 ,\sfv(2)^{\dot x_1-1}\sfu(1)\sfu(0)(\dot r_2x)\rangle=\langle \dot x_0+1 ,\widehat{\dot x_1-1} ,\sfu(1)\sfu(0)\sfu(\dot x_2)(\dot r_3x)\rangle\\
=\langle \dot x_0+1 ,\widehat{\dot x_1-1} ,\sfu(\dot x_2+1)(\dot r_3x)\rangle
=\langle \dot x_0+1 , \widehat{\dot x_1-1},\dot x_2+2 ,r_3x\rangle.
\end{multline}
Reciprocally, starting from a subtraction continued fraction, assuming that there exists $k\in\omega$ such that $x_i=2$ for $0< i< k$ and $x_{k}>2$, we calculate
\begin{multline}\label{regtosub}
x=\langle x_0 ,\widehat{k-1},x_{k} ,r_{k+1}x\rangle=\sfv( x_0)\sfv( 2)(r_2x)=\sfu(x_0-1)\sfu(1)\sfv(1)(r_2x)\\
=[x_0-1 ,\sfu(1)\sfv(1)(r_2x)]=[x_0-1 ,\sfu(1)\sfv(1)\sfv(2)(r_3x)]\\=[x_0-1 ,\sfu(1)\sfu(0)\sfu(1)\sfv(1)(r_3x)]
=[x_0-1 ,\sfu(2)\sfv(1)(r_3x)]\\
=\dots=[x_0-1 ,\sfu(k-1)\sfv(1)(r_{k}x)]
=[x_0-1 ,\sfu(k-1)\sfv(1)\sfv(x_{k})(r_{k+1}x)]\\
=[x_0-1 ,\sfu(k-1)\sfu(0)\sfu(1)\sfv(x_{k}-1)(r_{k+1}x)]
=[x_0-1 ,\sfu(k)\sfv(x_{k}-1)(r_{k+1}x)]\\
=[x_0-1 ,k ,\sfv(x_{k}-1)(r_{k+1}x)]=[x_0-1 ,k,x_{k}-2 ,\dot r_3x].
\end{multline}
Notice, that in the case when $x$ is a rational number, as soon as the input sequence in~\eqref{regtosub} stabilises to the constant sequence of $2$'s, the recursion halts producing new partial quotients in the simple continued fraction. This reflects the finiteness of  simple continued fractions of rational numbers~\cite{MR0146146}.
\begin{example}[The number $\pi$]
 Starting from the simple continued fraction
 \begin{equation}
\pi=[3 , 7, 15, 1, 292, 1, 1, 1, 2 , \dot r_{9}\pi]
\end{equation}
we derive the subtraction continued fraction as follows:
\begin{multline}
\pi=\sfu(3)\sfu(7)(\dot r_2\pi)=\sfv(4)\sfu(1)\sfu(6)(\dot r_2\pi)=\langle 4 ,\sfu(1)\sfu(6)(\dot r_2\pi)\rangle\\
=\langle 4 ,\sfv(2)^6\sfu(1)\sfu(0)(\dot r_2\pi)\rangle=\langle 4 ,\widehat{6} ,\sfu(1)\sfu(0)(\dot r_2\pi)\rangle
=\langle 4 ,\widehat{6} ,\sfu(16)\sfu(1)(\dot r_4\pi)\rangle\\
=\langle 4 ,\widehat{6},17 ,\sfu(1)\sfu(0)(\dot r_4\pi)\rangle
=\langle 4 ,\widehat{6},17 ,\sfu(293)\sfu(1)(\dot r_6\pi)\rangle\\
=\langle 4 ,\widehat{6},17,294 ,\sfu(1)\sfu(0)(\dot r_6\pi)\rangle
=\langle 4 ,\widehat{6},17,294 ,\sfu(2)\sfu(1)(\dot r_8\pi)\rangle\\
=\langle 4 ,\widehat{6},17,294,3 ,\sfu(3)(\dot r_{9}\pi)\rangle=\langle 4 ,\widehat{6},17,294,3,4 , r_{11}\pi\rangle
\end{multline}
with the corresponding sequence of right convergents 
\begin{equation}
 4,\ \frac72,\ \frac{10}3,\ \frac{13}4,\ \frac{16}5,\ \frac{19}6,\ \frac{22}7,\ \frac{355}{113}, 
 \frac{104348}{33215},
 \frac{312689}{99532}, \frac{1146408}{364913},\ \dots
\end{equation}
\end{example}
\begin{example}[The golden ratio] The golden ratio $\phi=\frac{\sqrt{5}+1}{2}$ is an irrational number with the simple continued fraction $\dot\phi_n=1$ for all $n\in\omega$. This corresponds to the equation $\phi=[1 ,\phi]$.
The following calculation shows that $\phi_0=2$ and $\phi_n=3$ for all $n>0$:
\begin{multline}
 \phi=\sfu(1)\sfu(1)(\phi)=\sfv(2)\sfu(1)\sfu(0)(\phi)\\=\langle 2 ,\sfu(1)\sfu(0)(\phi)\rangle=\langle 2 ,\sfu(2)\sfu(1)(\phi)\rangle
 =\langle 2 ,3 ,\sfu(1)\sfu(0)(\phi)\rangle=\dots
\end{multline}
In particular, we observe that the number 
\begin{equation}
\sfu(1)\sfu(0)(\phi)=\sfu(1)(1/\phi)=1+\phi
\end{equation}
solves the equation 
\begin{equation}
x=\langle 3 ,x\rangle\Leftrightarrow x_n=3\quad \forall n\in\omega.
\end{equation}
The sequence of right convergents of $\phi$ reads
\begin{equation}
2,\ \frac53,\ \frac{13}8,\ \frac{34}{21},\ \frac{89}{55},\ \frac{233}{144},\ \frac{610}{377},\ \frac{1597}{987},\ \frac{4181}{2584},\ \dots
\end{equation}
\end{example}
\begin{example}[The number $\log_23$] We have a  simple continued fraction
 \begin{equation}
\log_23=[1 ,1, 1, 2, 2, 3, 1, 5, 2 ,\dot r_9\log_23]
\end{equation}
which we transform into a subtraction continued fraction
\begin{multline}
 \log_23=\sfu(1)\sfu(1)(\dot r_2\log_23)=\langle 2 ,\sfu(1)\sfu(0)(\dot r_2\log_23)\rangle\\
 =\langle 2 ,\sfu(2)\sfu(2)(\dot r_4\log_23)\rangle
 =\langle 2 ,3,2 ,\sfu(1)\sfu(0)(\dot r_4\log_23)\rangle\\ =\langle 2 ,3,2 ,\sfu(3)\sfu(3)(\dot r_6\log_23)\rangle
 =\langle 2 ,3,2,4 ,\sfu(1)\sfu(2)(\dot r_6\log_23)\rangle\\ =\langle 2 ,3,2,4,\widehat{2} ,\sfu(1)\sfu(0)(\dot r_6\log_23)\rangle
 =\langle 2 ,3,2,4,\widehat{2} ,\sfu(2)\sfu(5)(\dot r_8\log_23)\rangle\\=\langle 2 ,3,2,4,\widehat{2},3 ,\sfu(1)\sfu(4)(\dot r_8\log_23)\rangle
 =\langle 2 ,3,2,4,\widehat{2},3,\widehat{4} ,\sfu(1)\sfu(0)(\dot r_8\log_23)\rangle\\ =\langle 2 ,3,2,4,\widehat{2},3,\widehat{4} ,\sfu(3)(\dot r_9\log_23)\rangle
 =\langle 2 ,3,2,4,\widehat{2},3,\widehat{4},4 ,r_{12}\log_23\rangle
\end{multline}
with the sequence of right convergents 
\begin{equation}
2,\ \frac53,\ \frac85,\ \frac{27}{17},\  \frac{46}{29},\ \frac{65}{41},\ \frac{149}{94},\ \frac{233}{147},\ \frac{317}{200},
\ \frac{401}{253},\ \frac{485}{306},\ \frac{1539}{971},\ \dots
\end{equation}

\end{example}
\def\cprime{$'$} \def\cprime{$'$}

%\bibliography{/Users/rinatkashaev/Dropbox/LatexStaff/biblio}{}
%\bibliographystyle{abbrv}

\end{document}